\documentclass[12pt]{amsart}
\usepackage{times}
\usepackage{amsfonts}
\usepackage{amssymb}
\usepackage{bbm}
\usepackage{times}
\usepackage{amssymb}
\usepackage{amscd}
\usepackage{graphicx}

\usepackage{amsmath}
\usepackage{amssymb}
\usepackage{amsmath,amsthm,amscd,amssymb}
\usepackage{vmargin,cite}
\usepackage{graphicx}
\usepackage{subfigure}
\usepackage{color}
\usepackage{mathrsfs}
\usepackage{amsmath}
\usepackage{amsfonts}
\usepackage{amscd}
\usepackage{latexsym}
\usepackage{amsthm}
\usepackage{amssymb}
\usepackage{amsmath}
\usepackage{setspace}
\theoremstyle{plain}
\newtheorem{theorem}{Theorem}[section]

\newtheorem{lemma}[theorem]{Lemma}

\newtheorem{definition}[theorem]{Definition}

\def\ep{\varepsilon}

\begin{document}

\vskip 0.5cm

\title[non-unital tracially approximated ${\rm C^*}$-algebras] {non-unital tracially approximated ${\rm C^*}$-algebras}
\author{Qingzhai Fan, Chengyu Long and Shan Zhang}

\address{Qingzhai Fan\\ Department of Mathematics\\  Shanghai Maritime University\\
Shanghai\\China
\\  201306 }
\email{qzfan@shmtu.edu.cn}
\address{Chenyu Long\\ Department of Mathematics\\  Shanghai Maritime University\\
Shanghai\\China
\\  201306 }
\email{longchengyu321@163.com}

\address{Shan Zhang\\ Department of Mathematics\\  Shanghai Maritime University\\
Shanghai\\China
\\  201306 }
\email{202031010014@stu.shmtu.edu.cn}

\thanks{{\bf Key words}  ${\rm C^*}$-algebras, tracial approximation, tracially $\mathcal{Z}$-absorbing.}

\thanks{2000 \emph{Mathematics Subject Classification\rm{:}} 46L35, 46L05, 46L80}

\begin{abstract} In this paper, we  introduce  a class of  non-unital  tracial approximation  ${\rm C^*}$-algebras.  Consider the class of  ${\rm C^*}$-algebras which   are tracially $\mathcal{Z}$-absorbing (in the  sense of Amint, Golestani, Jamali, Phillips's  simple tracially $\mathcal{Z}$-absorbing or Castillejos, Li,  Szabvo's    tracial $\mathcal{Z}$-stability). Then $A$  is  tracially $\mathcal{Z}$-absorbing for  any  simple  ${\rm C^*}$-algebra $A$ in the corresponding  class of non-unital tracial approximation ${\rm C^*}$-algebras.
\end{abstract}

\maketitle

\section{Introduction}
The Elliott program for the classification of amenable
 ${\rm C^*}$-algebras might be said to have begun with the ${\rm K}$-theoretical
 classification of $\rm AF$ algebras in \cite{E1}. A major next step was the classification of simple
 $\rm AH$ algebras without dimension growth (in the real rank zero case see \cite{E6}, and in the general case
 see \cite{GL}).
 A crucial intermediate step was Lin's axiomatization of  Elliott-Gong's decomposition theorem for  simple $\rm AH$ algebras of real rank zero (classified by Elliott-Gong in \cite{E6}) and Gong's decomposition theorem (\cite{G1}) for simple $\rm AH$ algebras (classified by Elliott-Gong-Li in \cite{GL}).  Lin introduced the concepts of   $\rm TAF$ and $\rm TAI$ (\cite{L0} and \cite{L1}). Instead of assuming inductive limit structure, Lin started with a certain abstract (tracial) approximation property.  This led eventually to the classification of  simple separable amenable ${\rm C^*}$-algebras with finite nuclear dimension in the UCT class (see \cite{EGLN1}, \cite{EZ5}, \cite{GLN1}, \cite{GLN2}, \cite{TWW1},  \cite {GL2}).

 Inspired by   Lin's  tracial interval algebras in \cite{L1},  Elliott and  Niu in \cite{EZ} considered  the natural   notion of
 tracial approximation by  other classes of ${\rm C^*}$-algebras. Let $\Omega$ be a class of unital ${\rm C^*}$-algebras. Then the class
 of simple  separable  ${\rm C^*}$-algebras which can be tracially
approximated by ${\rm C^*}$-algebras in $\Omega,$ denoted by  ${\rm TA}\Omega$,  is
defined as follows.  A  simple unital ${\rm C^*}$-algebra $A$ is said to
      belong to the class ${\rm TA}\Omega$ if,  for any
 $\ep>0,$ any finite
subset $F\subseteq A,$ and any non-zero element $a\geq 0,$  there
are a  projection $p\in A$ and a ${\rm C^*}$-subalgebra $B$ of $A$ with
$1_B=p$ and $B\in \Omega$ such that

$(1)$ $\|xp-px\|<\ep$ for all $x\in  F$,

$(2)$ $pxp\in_\ep B$ for
all $x\in  F$, and

 $(3)$ $1-p$ is Murray-von Neumann equivalent to a
projection in $\overline{aAa}$.

The question of  which  properties pass  from a class $\Omega$ to the class ${\rm TA}\Omega$ is
interesting and sometimes important. In fact,  the property of being of stable rank one, and the
property that the strict order on projections is determined by
traces,  are important in the classification theorem in \cite{NNN}.
  In  \cite{EZ}, Elliott and Niu and in \cite{EFF} and \cite{Q8},  Elliott, Fan, and Fang  showed that certain properties of  ${\rm C^*}$-algebras in a class $\Omega$ are inherited by a simple unital ${\rm C^*}$-algebra
in the class ${\rm TA}\Omega$.

Inspired by   Lin's  tracial approximation  and by Castillejos, Li,  Szabvo's    tracial $\mathcal{Z}$-stability of simple non-unital ${\rm C^*}$-algebras in \cite{CKS} and Amint, Golestani, Jamali, Phillips's  simple tracially $\mathcal{Z}$-absorbing ${\rm C^*}$-algebras in \cite{AGJC} and by  Forough,  Golestani's  ${\rm C^*}$-algebras with Property $(T_0)$ in \cite{FG}.

In this paper, we introduce a class of non-unital tracial approximation ${\rm C^*}$-algebras.
 We show that  the following  result:
 let $\Omega$ be a class of   ${\rm C^*}$-algebras which   are tracially $\mathcal{Z}$-absorbing (in the  sense of Amint, Golestani, Jamali, Phillips's  simple tracially $\mathcal{Z}$-absorbing or Castillejos, Li,  Szabvo's    tracial $\mathcal{Z}$-stability).  Then  $A$  is   tracially $\mathcal{Z}$-absorbing
  for  any  simple  ${\rm C^*}$-algebra $A$  which is  non-unital tracially
approximated by ${\rm C^*}$-algebras in $\Omega$.

\section{Preliminaries and definitions}

Let $A$ be a ${\rm C^*}$-algebra, and let ${\rm M}_n(A)$ denote the ${\rm C^*}$-algebra of  $n\times n$ matrices with entries
elements of $A$. Let ${\rm M}_{\infty}(A)$ denote the algebraic inductive  limit of the sequence  $({\rm M}_n(A),\phi_n),$
where $\phi_n:{\rm M}_n(A)\to {\rm M}_{n+1}(A)$ is the canonical embedding as the upper left-hand corner block.
 Let ${\rm M}_{\infty}(A)_+$ (respectively, ${\rm M}_{n}(A)_+$) denote
the positive elements of ${\rm M}_{\infty}(A)$ (respectively, ${\rm M}_{n}(A)$).  Given $a, ~b\in {\rm M}_{\infty}(A)_+,$
one says  that $a$ is Cuntz subequivalent to $b$ (written $a\precsim b$) if there is a sequence $(v_n)_{n=1}^\infty$
of elements of ${\rm M}_{\infty}(A)$ such that $$\lim_{n\to \infty}\|v_nbv_n^*-a\|=0.$$
One says that $a$ and $b$ are Cuntz equivalent (written $a\sim b$) if $a\precsim b$ and $b\precsim a$. We  shall write $\langle a\rangle$ for the equivalence class of $a$.

The object ${\rm W}(A):={\rm M}_{\infty}(A)_+/\sim$
 will be called the Cuntz semigroup of $A$. (See \cite{CEI}.)  Observe that  any $a, b\in {\rm M}_{\infty}(A)_+$
are Cuntz equivalent  to orthogonal  elements $a', b'\in {\rm M}_{\infty}(A)_+$ (i.e., $a'b'=0$),   and so ${\rm Cu}(A)$ becomes  an ordered  semigroup   when equipped with the addition operation
$$\langle a\rangle+\langle b\rangle=\langle a+ b\rangle$$
 whenever $ab=0$, and the order relation
$$\langle a\rangle\leq \langle b\rangle\Leftrightarrow a\precsim b.$$

Given $a$ in ${\rm M}_{\infty}(A)_+$ and $\varepsilon>0,$ we denote by $(a-\varepsilon)_+$ the element of ${\rm C^*}(a)$ corresponding (via the functional calculus) to the function $f(t)={\max (0, t-\varepsilon)},~~ t\in \sigma(a)$. By the functional calculus, it follows in a straightforward manner that $((a-\varepsilon_1)_+-\varepsilon_2)_+=(a-(\varepsilon_1+\varepsilon_2))_+.$

\begin{definition}(\cite{AGJC}, \cite{CKS}.)\label{def:2.7}  A simple  ${\rm C^*}$-algebra $A$ is tracially $\mathcal{Z}$-absorbing, if $A\neq {{\mathbb{C}}}$
and for any finite set $F\subseteq A,$ $\varepsilon>0,$   non-zero positive element $a,~ b\in A$,
and $n\in {{\mathbb{N}}},$ there is a  completely positive order zero  contraction $\psi: M_n\to A$, where completely positive order zero map means preserving  orthogonality, i.e.,
 $\psi(e)\psi(f)=0$ for all $e,~f\in {\rm M}_n$ with $ef=0$, such that
the following properties hold:

$(1)$ $(b^2-b\psi(1)b-\varepsilon)_+\precsim a,$ and

$(2)$ for any normalized element $x\in {\rm M}_n$ (i.e., with $\|x\|=1$) and any $y\in F$ we have $\|\psi(x)y-y\psi(x)\|<\varepsilon$.
\end{definition}

The following theorem is Theorem 4.1 of  {\cite{AGJC}}.
\begin{theorem}\label{thm:2.11}
Let $A$ be a simple tracially $\mathcal{Z}$-absorbing ${\rm C^*}$-algebra and let
$B$ be a hereditary ${\rm C^*}$-subalgebra of $A$. Then $B$ is also  tracially $\mathcal{Z}$-absorbing
\end{theorem}

\begin{definition}(\cite{HO}.)\label{def:2.8} We say a unital ${\rm C^*}$-algebra $A$ is tracially $\mathcal{Z}$-absorbing, if $A\neq {{\mathbb{C}}}$
and for any finite set $F\subseteq A,$ $\varepsilon>0,$   non-zero positive element $a\in A$,
and $n\in {{\mathbb{N}}},$ there is a  completely positive order zero  contraction $\psi: M_n\to A$, where completely positive order zero map means preserving  orthogonality, i.e.,
 $\psi(e)\psi(f)=0$ for all $e,f\in {\rm M}_n$ with $ef=0$, such that
the following properties hold:

$(1)$ $1-\psi(1)\precsim a,$ and

$(2)$ for any normalized element $x\in {\rm M}_n$ (i.e., with $\|x\|=1$) and any $y\in F$ we have $\|\psi(x)y-y\psi(x)\|<\varepsilon$.
\end{definition}

Remark: By \cite{AGJC} or  \cite{CKS}, if $A$ is a unital simple  ${\rm C^*}$-algebra, we know that Definition \ref{def:2.8} is equivalent to Definition \ref{def:2.7}.

Let $\Omega$ be a class of  ${\rm C^*}$-algebras. Then the class
 of simple unital  separable  ${\rm C^*}$-algebras which can be tracially
approximated by ${\rm C^*}$-algebras in $\Omega$, denoted by  ${\rm TA}\Omega$,  is
defined as follows.

\begin{definition}(\cite{EZ}.) A  simple unital ${\rm C^*}$-algebra $A$ is said to
      belong to the class ${\rm TA}\Omega$, if for any
 $\ep>0,$ any finite
subset $F\subseteq A,$ and any non-zero element $a\geq 0,$  there
are a  projection $p\in A$, and a ${\rm C^*}$-subalgebra $B$ of $A$ with
$1_B=p$ and $B\in \Omega$, such that

$(1)$ $\|xp-px\|<\ep$ for all $x\in  F$,

$(2)$ $pxp\in_\ep B$ for
all $x\in  F$, and

 $(3)$ $1-p\precsim a$.
\end{definition}

Let $\Omega$ be a class of ${\rm C^*}$-algebras. Then the class
 of simple  separable  ${\rm C^*}$-algebras which can be strongly tracially
approximated by ${\rm C^*}$-algebras in $\Omega$, denoted by  ${\rm STA}\Omega$,  is
defined as follows.

\begin{definition}\label{def:2.9}
Let $A$ be a simple ${\rm C^*}$-algebra. $A$ will said to
      belong to the class ${\rm STA}\Omega$, if for any $\varepsilon>0$, every finite set $F\subseteq A$, and for every positive elements $a,~ y\in A_{+}$ with $a\neq 0$,  and a ${\rm C^*}$-subalgebra $B$ of $A$ and $B\in \Omega$,  and a projection $p\in B$ such that the following hold:

$(1)$ $\|xp-px\|<\ep$ for all $x\in  F$,

$(2)$ $pxp\in_\ep B$ for
all $x\in  F$, and

 $(3)$ $(y^{2}-ypy-\varepsilon)_{+}\precsim _{A} a$.
\end{definition}

The proof of the following lemma  is similar with Lemma 3.6.5 in \cite{L1}.

\begin{lemma}\label{lem:2.6}
If the class $\Omega$ is the closed  under passing to unital hereditary ${\rm C^*}$-subalgebras, then the class $\rm {STA}\Omega$ is closed  under passing to unital hereditary ${\rm C^*}$-subalgebras.
\end{lemma}

\begin{proof}  We show that  $B=qAq \in \rm {STA}\Omega$ for $C^{*}$-algebra $A\in \rm {STA}\Omega$ and  some projection $q\in A$.

Let ${F}\subset B$ be a finite subset, we may assume that  $F$ is a subset of the unit ball of $B$, any $\varepsilon>0$,   and for every positive elements $a,~ y\in A_{+}$ with $a\neq 0$,  we  will show that there exist a  ${\rm C^*}$-subalgebra $D$ of $B$ and $D\in \Omega$,  and a projection $p\in D$ such that the following hold:

$(1)$ $\|xp-px\|<\ep$ for all $x\in  F$,

$(2)$ $pxp\in_\ep D$ for
all $x\in  F$, and

 $(3)$ $(y^{2}-yp'y-\varepsilon)_{+}\precsim a$.

 Since $A\in \rm {STA}\Omega$, for ${G}={F}\cup \{q\}$,  $\delta>0$,  and  positive elements $a, y\in A_{+}$ with $a\neq 0$, there are a projection $p\in A$ and a $C^{*}$-subalgebra $C$ of $A$ and $C\in \Omega$ such that

$(1)'$ $\parallel p'x-xp'\parallel<\delta$ for all $x\in {G},$

$(2)'$ $p'xp'\in_{\delta}C$ for all $x\in {G}$ and

$(3)'$ $(y^{2}-yp'y-\delta)_{+}\precsim a$.

By $(1)'$ and $(2)'$, one has  $$\|(qp'q)^{2}-qp'q\|<2\delta,$$

With small $\delta$, we obtain a projection $e\in C$, a projection $p\in qAq = B$ and a unitary $u\in A$ such that
$$\| p-qp'q\| <\frac{\varepsilon}{3},\| p-e\| <\frac{\varepsilon}{3},u^{*}eu=p~and~\| u-1\| <\frac{\varepsilon}{3}.$$
Let $C_{1} = u^{*}eCeu$. Since  $eCe\in \Omega$, one has $C_{1}\in \Omega$.

For each $x\in {F}$
$$\| px-xp\|\leq \| qp'q-p\|\| x\|+\|qp'qx-xqp'q\|+\|x\|\|qp'q-p\|< \frac{\varepsilon}{3}+\delta+\frac{\varepsilon}{3}<\varepsilon.$$
For any  $x\in {F}$, there is a $c_{x}\in C$ such that
$$\| p'qxqp'-c_{x}\|< 3\delta.$$
Thus,

$$pxp\approx_{6\delta}pp'qxqp'p\approx_{6\delta}pc_xp\approx_{\frac{\varepsilon}{3}}
ec_{x}e\approx_{\varepsilon}u^*ec_{x}eu.$$
Therefore we have
$$pxp\in_{2\varepsilon}C_{1}~for~all~x\in{F}$$
Since $\|(q-qp'q)-(q-p)\|< 1,$~ if $\delta$ is sufficiently small, $(q-p)(q-qp'q)(q-p)$ is invertible in $(q-p)A(q-p).$
Hence
$$q-p\sim (q-p)(q-qp'q)(q-p)=(q-p)q(1-p')q(q-p) \precsim (q-p)(1-p')(q-p)\precsim q(1-p')q.$$
Also, it follows that $q(1-p')q\precsim 1-p'.$ Finally, we have
$$q-p\precsim q(1-p')q \precsim 1-p'\precsim a.$$
Thus $B \in \rm{STA}\Omega$.
\end{proof}

\begin{theorem}\label{thm:2.10}
Let $\Omega$  be a class  of ${\rm C^*}$-algebras such that  $\Omega$ is the closed  under passing to unital hereditary ${\rm C^*}$-subalgebras. Let $A$ be a unital simple ${\rm C^*}$-algebra, the following statements are equivalent:

$(I)$ $A$  belong to the class of $\rm {TA}\Omega$ in the sense of definition $2.5.$

$(II)$ $A$ belong to  the class of  $\rm {STA}\Omega$  in the sense of definition$2.9.$
\end{theorem}
\begin{proof}
 $(II)\Rightarrow (I)$, we need to show that
for any
 $\ep>0,$ any finite
subset $F\subseteq A,$ and any non-zero element $a\geq 0,$  there
are a  projection $p\in A$, and a ${\rm C^*}$-subalgebra $D$ of $A$ with
$1_D=p$ and $D\in \Omega$, such that

$(1)$ $\|xp-px\|<\ep$ for all $x\in  F$,

$(2)$ $pxp\in_\ep D$ for
all $x\in  F$, and

 $(3)$ $1-p\precsim a$.

Since $A\in \rm {STA}\Omega$, for given  $\varepsilon>0$,  finite set $F\subseteq A$, and for every positive elements $a, 1\in A_{+}$ with $a\neq 0$,  there exist a a ${\rm C^*}$-subalgebra $B$ of $A$ and $B\in \Omega$,  and a projection $p\in B$ such that the following hold:

$(1)'$ $\|xp-px\|<\ep$ for all $x\in  F$,

$(2)'$ $pxp\in_\ep B$ for
all $x\in  F$, and

 $(3)'$ $(1-p-\varepsilon)_{+}=1-p\precsim a$.

 We take $D=pBp$, then $1_D=p$, and $D\in \Omega$, such that

 $(1)$ $\|xp-px\|<\ep$ for all $x\in  F$,

$(2)$ $pxp\in_\ep D$ for
all $x\in  F$, and

 $(3)$ $1-p\precsim a$.

 $(I)\Rightarrow (II)$,we need to show that
 for any $\varepsilon>0$, every finite set $F\subseteq A$, and for every positive elements $a,~ y\in A_{+}$ with $a\neq 0$,  and a ${\rm C^*}$-subalgebra $B$ of $A$ and $B\in \Omega$,  and a projection $p\in B$ such that the following hold:

$(1)$ $\|xp-px\|<\ep$ for all $x\in  F$,

$(2)$ $pxp\in_\ep B$ for
all $x\in  F$, and

 $(3)$ $(y^{2}-ypy-\varepsilon)_{+}\precsim _{A} a$.

 Since $A\in \rm {TA}\Omega$, for given  $\varepsilon>0$,
 and elements  $y,~ a\in A_+$ and $a\geq 0,$   there
are a  projection $p\in A$, and a ${\rm C^*}$-subalgebra $D$ of $A$ with
$1_D=p$ and $D\in \Omega$, such that

$(1)'$ $\|xp-px\|<\ep$ for all $x\in  F$,

$(2)'$ $pxp\in_\ep D$ for
all $x\in  F$, and

 $(3)'$ $1-p\precsim a$.

By $(3)'$ we have

$(3)$ $(y^2-ypy-\varepsilon)_+\precsim y(1-p)y\precsim 1-p\precsim a$.

\end{proof}

\section{The main results}

\begin{theorem}\label{thm:3.1}
 Let $\Omega$ be a class of ${\rm C^*}$-algebras which are  tracially $\mathcal{Z}$-absorbing. Then $A$ is  tracially $\mathcal{Z}$-absorbing  for  any  infinite-dimensional simple ${\rm C^*}$-algebra $A\in {\rm STA}\Omega$.
\end{theorem}
\begin{proof}
We need to show that
for any finite set $F=\{a_1,~a_2,~\cdots,~ a_k\}\subseteq A,$  any $\varepsilon>0,$  any non-zero positive elements $a, b \in A$
and $n\in {{{\mathbb{N}}}},$  there is an order zero contraction $\psi:{\rm M}_n\to A$ such that
the following conditions hold:

$(1)$ $(b^2-b\psi(1)b-\varepsilon)_+\precsim a$, and

$(2)$ for any normalized element $x\in {\rm M}_n$ and any $y\in F,$ we have
$\|\psi(x)y-y\psi(x)\|<\varepsilon.$

Since $A$ is  an  infinite-dimensional simple ${\rm C^*}$-algebra, there exist non-zero positive elements $\bar{a}, \bar{\bar{{a}}} \in A_+$, such that $\bar{a}\sim \bar{\bar{{a}}}, \bar{a}\bar{\bar{{a}}}=0$ and $\bar{a}+\bar{\bar{{a}}}\precsim a$.

Since $A\in {\rm STA}\Omega,$ for $F=\{a_1,~a_2,~\cdots, ~a_k\}$, any $\varepsilon'>0,$  non-zero positive elements $\bar{a},b \in A$,
there  are   a ${\rm C^*}$-subalgebra $B$ of $A$ and a projection $p\in B$  with   $B\in \Omega$  such that

$(1)'$ $\|a_ip-pa_i\|<\ep',$  for all $1\leq i\leq k,$  $\|bp-pb\|<\ep',~ \|\bar{a}p-p\bar{a}\|<\ep',$

$(2)'$ $pa_ip\in_{\ep'} B,$  for all $1\leq i\leq k,$  $pbp\in_{\ep'} B,~ p\bar{a}p\in_{\ep'} B$ and

$(3)'$ $(b^2-bpb-\varepsilon')_+\precsim \bar{a}.$

When is $\varepsilon'$ small enough, by $(1)'$ and $(2)'$,   there exist $\bar{a}', ~a_1',~ a_2',~\cdots,$ $ ~a_k', ~b'\in B$  such that
$$\|b'-pb'p\|<4\varepsilon', \|\bar{a}'-p\bar{a}p\|<4\varepsilon', \|pa_ip-a_i'\|<4\varepsilon',$$

 Since $B\in \Omega,$ by Lemma \ref{lem:2.6} and by Theorem \ref{thm:2.11}, one has $pBp\in \Omega$,   for  $G=\{a_1',~a_2',~\cdots, ~a_k'\}\subseteq B,$  $\varepsilon'>0$ as specified,  the non-zero positive element $\bar{a}'\in B$,
and $n\in {{\mathbb{N}}},$  there  is  an order zero contraction $\psi: {\rm M}_n\to pBp\subseteq A$  with
the following properties:

$(1)''$ $p-\psi(1)\precsim \bar{a}',$ and

$(2)''$ for any normalized element $x\in {\rm M}_n$ and any $a_i'\in G,$  we have $$\|\psi(x)a_i'-a_i'\psi(x)\|<\varepsilon'.$$

By $(1)'$, $$\|a_i-pa_ip-(1-p)a_i(1-p)\|<\varepsilon'.$$

 Therefore, for any normalized element $x\in {\rm M}_n$,
we have
\begin{eqnarray}
\label{Eq:eq1}
  && (2)~~~ \|\psi(x)a_i-a_i\psi(x)\|\nonumber\\
  &&\leq\|\psi(x)a_i-\psi(x)(pa_ip+(1-p)a_i(1-p))\|\nonumber\\
&&+\|\psi(x)(pa_ip+(1-p)a_i(1-p))
-\psi(x)(a_i'+(1-p)a_i(1-p))\|\nonumber\\
&&+ \|\psi(x)(a_i'+(1-p)a_i(1-p))-(a_i'+(1-p)a_i(1-p))\psi(x)\|\nonumber\\
 &&+ \|(a_i'+(1-p)a_i(1-p))\psi(x)-(pa_ip+(1-p)a_i(1-p))\psi(x)\|\nonumber\\
 &&+\|(pa_ip+(1-p)a_i(1-p))\psi(x)
-a_i\psi(x)\|\nonumber\\
&&<2\varepsilon'+\varepsilon'+2\varepsilon'+\varepsilon'+3\varepsilon'<13\varepsilon'<\varepsilon.\nonumber
  \end{eqnarray}

Since $p-\psi(1)\precsim \bar{a}',$ one has $b(p-\psi(1))b\precsim \bar{a}'$, therefore, one has

$(1)$ $(b^2-b\psi(1)b-2\varepsilon)_+\precsim (b^2-bpb-\varepsilon')_+ \oplus(bpb-b\psi(1)b)\precsim \bar{\bar{a}}\oplus \bar{{a}}'\precsim \bar{\bar{a}}+\bar{a}\precsim a.$

\end{proof}

 \end{document}